\documentclass[11pt, a4paper, reqno]{article}
\usepackage[english]{babel}
\usepackage[utf8]{inputenc}
\usepackage{amsmath}
\usepackage{amssymb}
\usepackage{amsthm}
\usepackage{graphicx}
\usepackage{setspace}
\usepackage{yfonts}
\usepackage{mathtools}
\numberwithin{equation}{section}
\usepackage{xcolor}
\usepackage{tikz}
\usepackage{parskip}
\usepackage{fullpage}
\usepackage{array}
\usepackage{hyperref}
\usepackage{subfig}
\usepackage{ytableau}
\newcommand{\ds}{\displaystyle}
\newcommand{\Z}{\mathbb{Z}}

\newcommand{\N}{\mathbb{N}}

\newcommand{\ext}{\text{ext}}

\newcommand{\s}{\hspace{0.08cm}}

\usepackage[numbib, nottoc]{tocbibind}

\newtheorem{theorem}{Theorem}[section]

\newtheorem{proposition}[theorem]{Proposition}
\newtheorem{conjecture}[theorem]{Conjecture}

\theoremstyle{definition}

\newtheorem{example}[theorem]{Example}

\newtheorem*{remark*}{Remark}
\usepackage[a4paper,top=4cm,bottom=4cm,left=1.5cm,right=1.5cm,marginparwidth=1.75cm]{geometry}

\makeatletter
\def\@settitle{
	\begin{center}
		\baselineskip\p@
		\normalfont
		\@title
	\end{center}
}
\makeatother

\makeatletter
\def\Ddots{\mathinner{\mkern1mu\raise\p@
\vbox{\kern7\p@\hbox{.}}\mkern2mu
\raise4\p@\hbox{.}\mkern2mu\raise7\p@\hbox{.}\mkern1mu}}
\makeatother

\date{}
\title{\textbf{Enumerating Permutations and Rim Hooks Characterized by Double Descent Sets}}
\author{\textbf{Christopher Zhu}\\\textbf{The Roxbury Latin School}}

\begin{document}
	\maketitle
	\vspace{-0.7cm}

	\vspace{0.5cm}
	
	\begin{center}
	    \textbf{Abstract}
	\end{center}
	
	 Let $dd(I;n)$ denote the number of permutations of $[n]$ with double descent set $I$. For singleton sets $I$, we present a recursive formula for $dd(I;n)$ and a method to estimate $dd(I;n)$. We also discuss the enumeration of certain classes of rim hooks. Let $\mathcal{R}_I(n)$ denote the set of all rim hooks of length $n$ with double descent set $I$, so that any tableau of one of these rim hooks corresponds to a permutation with double descent set $I$. We present a formula for the size of $\mathcal{R}_I(n)$ when $I$ is a singleton set, and we also present a formula for the size of $\mathcal{R}_I(n)$ when $I$ is the empty set. We additionally present several conjectures about the asymptotics of certain ratios of $dd(I;n)$.
	
	\vspace{1cm}
	
	\setstretch{0.55}\tableofcontents
	
	\setstretch{1.4}
	
	\section{Introduction}
	 Throughout this paper, we let $I$ be a finite set of positive integers. We will also use the standard notation $[n]$ to represent the set $\{1, 2, ..., n\}$, and for $m<n$, we let $[m, n]$ represent the set $\{m, m+1, ..., n\}$.
	
	 Consider the symmetric group $\mathfrak{S}_n$ of permutations $w = [w_1,w_2,...,w_n]$ of $[n]$. A \textit{descent} of $w$ is an index $i$ satisfying $w_i > w_{i+1}$, and the \textit{descent set} of $w$ is $$\text{Des}(w) := \{i\s|\s i\text{ is a descent of }w\} \subseteq [n-1].$$ For example, $\text{Des}([1,7,3,2,6,4,5]) = \{2, 3, 5\}$. Next, consider the set of all permutations of $[n]$ with a given descent set, $$D(I;n) := \{w \in \mathfrak{S}_n \s|\s \text{Des}(w) = I \},$$ and its cardinality $$d(I;n) := \#D(I;n).$$ 
	
	 Using the Principle of Inclusion and Exclusion, MacMahon \cite{mac} proved in 1915 that $d(I;n)$ is in fact a polynomial in $n$, for a fixed finite set $I$. We call $d(I;n)$ the \textit{descent polynomial} of $I$. For the next century, little detailed work was done on these descent polynomials, until Diaz-Lopez et al. \cite{DLHIOS} published a paper on them in 2017. In that paper, Diaz-Lopez et al. provide recursions for $d(I;n)$ and extensively study algebraic properties of descent polynomials. Some of their results include a theorem about the positivity of coefficients of $d(I;n)$ when expressed in a Newton basis, as well as bounds on roots of descent polynomials.

	 Along the lines of descents, we can also define a \textit{peak} of a permutation $w$ as an index $i$ satisfying $w_{i-1} < w_i > w_{i+1}$. Analogously, we can define the peak set of a permutation $w$ as $$\text{Peak}(w) := \{i \s|\s i \text{ is a peak of }w\} \subseteq [2,n-1].$$ Following this definition is $P(I;n) := \{w \in \mathfrak{S}_n \s|\s \text{Peak}(w) = I \}$. In 2013, Billey et al. \cite{bbs} studied $\#P(I;n)$ as a function of $n$ for fixed $I$ and showed that in general it is not polynomial, but of the form $p(I;n)2^{n-\#I-1}$, where $p(I;n)$ is a polynomial. This is called the \textit{peak polynomial} of $I$. Billey et al. also presented a recursion for $p(I;n)$, and studied formulas for $p(I;n)$ given a specific set $I$.
	
	 We now move on to \textit{double descents}, which we investigate in this paper. A double descent of a permutation $w$ is an index satisfying $w_{i-1} > w_i > w_{i+1}$. Next, we define $$\text{DDes}(w) := \{i \s|\s i \text{ is a double descent of }w \} \subseteq [2,n-1],$$ and analogously, $$DD(I;n) := \{w\in \mathfrak{S}_n \s|\s \text{DDes}(w) = I \},$$ and $dd(I;n) := \#DD(I;n)$. For example, $DD(\{2\}; 4) = \{[3,2,1,4], [4,3,1,2], [4,2,1,3]\}$, so $dd(\{2\}; 4) = 3$. 
	
	 The paper is structured as follows. We start off in Section \ref{woddes} where we discuss known results about permutations without double descents. After that, we discuss permutations with singleton double descent sets in Section \ref{singletonddes}. In particular, we present a recursion for $dd(I;n)$ for singleton $I = \{k\}$, which allows us to express $dd(\{k\};n)$ in terms of $dd(\{l\};m)$ for $l<k$ and $m<n$. We also discuss a method for estimating values of $dd(I; n)$ again for singleton sets $I$. In the next Section (\ref{rimhooks}), we analyze certain classes of rim hooks associated with singleton and empty double descent sets, and we also provide theorems regarding the sizes of these classes of rim hooks. While discussing rim hooks, we develop the theory of \textit{minimal elements}, which is useful in several proofs. Afterwards, we quickly take a look at \textit{circular permutations} in Section \ref{other}, another permutation-associated object (just like rim hooks). Then, in Section \ref{conjs}, we bring up conjectures obtained from studying patterns in computer-generated data. Most importantly, we discuss a conjecture that highlights a large difference between descents and double descents, as well as the so-called ``down up down up" conjecture which reveals an interesting pattern in data concerning singleton double descent sets. Finally, we conclude with a section on future research questions.

	\section{Permutations Without Double Descents}\label{woddes}
	
	 In this section, we begin our discussion of permutations and double descents by discussing current results in the literature. We start off by considering the specific case of permutations with no double descents and no initial descent, which will build up to permutations with no double descents in general. That is, we are considering all $w \in \mathfrak{S}_n$ such that $\text{DDes}(w) = \emptyset$ and $w_1 < w_2$. We will use $b_n$ to denote the number of such permutations in $\mathfrak{S}_n$. On OEIS \cite{b_n}, Michael Somos presents the following recursion for the sequence $b_n$, which is useful for finding a generating function for $b_n$. 
	\begin{proposition}[\cite{somos}]\label{2.1}
		 The function $b_n$ satisfies the following recursion: $$b_{n+1} = \ds \Bigg(\sum_{k=0}^n\dbinom{n}{k}b_kb_{n-k}\Bigg) - b_n.$$
	\end{proposition}

	On the same OEIS reference to Somos' recurrence, Peter Bala provides an  exponential generating function for $b_n$. This is useful for computing $dd(\emptyset; n)$.
	
	\begin{proposition}[\cite{bala}]\label{2.2}
		The exponential generating function for $b_n$ is $\dfrac{1}{2} + \dfrac{\sqrt{3}}{2}\tan\left(\dfrac{\sqrt{3}}{2}x + \dfrac{\pi}{6}\right)$.
	\end{proposition}
	
	The following recursion, which relates $dd(\emptyset; n)$ and $b_n$, is given by Emanuele Munarini on OEIS \cite{a(n)}.
	
	\begin{proposition}[\cite{munarini}]
		The function $dd(\emptyset, n)$ satisfies the following recursion: $$dd(\emptyset; n+1) = \ds\sum_{k=0}^n\dbinom{n}{k}\cdot dd(\emptyset, k)\cdot b_{n-k}.$$
	\end{proposition}

	This recursion, along with Proposition \ref{2.2}, can be used to prove the formula for the exponential generating function of $dd(\emptyset; n)$ given by Noam Elkies on OEIS \cite{a(n)}.
	
	\begin{proposition}[\cite{elkies}]
		The exponential generating function for $dd(\emptyset; n)$ is $\dfrac{\frac{\sqrt{3}}{2}\cdot e^{\frac{x}{2}}}{\cos\left(\frac{\sqrt{3}}{2}x + \frac{\pi}{6}\right)}$.
	\end{proposition}

	These results provide most of the background on permutations whose double descent set is the empty set. We now proceed to study permutations which have singleton double descent sets.
	
	\section{Singleton Double Descent Sets}\label{singletonddes}
	 The main enumeration theorem of this section is the following recursion for $dd(I;n)$ when $I$ is a singleton set.
	
	\begin{theorem}\label{recursion1}
		Let $I = \{m\}$ be a singleton set. Then we have 
		\begin{equation}\label{eq1}
		\begin{split}	
		dd(I; n+1) = & \Bigg(\sum_{k = m+1}^n \dbinom{n}{k}\cdot dd(I; k)\cdot b_{n-k}\Bigg) \\
		& + \dbinom{n}{m-2}\cdot dd(\emptyset;m-2)\cdot \big(dd(\emptyset;n-m+2) - b_{n-m+2}\big) \\
		& + \Bigg(\sum_{k=0}^{m-4}\dbinom{n}{k}\cdot dd(\emptyset;k)\cdot 
		c(\{m-1-k\}; n-k)\Bigg)
		\end{split}
		\end{equation}
		where $c(I; n)$ denotes the number of permutations in $\mathfrak{S}_n$ with an initial ascent and with double descent set $I$.
	\end{theorem}
	
	\begin{proof}
		To construct a permutation $w \in \mathfrak{S}_{n+1}$ with a double descent at $m$, we first consider possible values of $w^{-1}(n+1)$. Because there is a double descent at $m$, we have $w_{m-1} > w_m > w_{m+1}$, so $w^{-1}(n+1) \notin $ $\{m, m+1\}$ because all other $w_i < n+1$. Also, $w^{-1}(n+1)\neq m-2$; otherwise, there would be a double descent at $w_{m-1}$ since we have $w_{m-1} > w_m$. Thus, $w^{-1}(n+1) \in [m+2, n+1] \cup \{m-1\} \cup [m-3]$. Suppose $w^{-1}(n+1) \in [m+2, n+1]$. Then we can choose $k \in [m+1, n]$ elements of $[n]$ to form a permutation to the left of $n+1$ with a double descent at $m$, and the remaining $n-k$ elements of $[n]$ form a permutation to the right of $n+1$ with no initial descent and no double descents. For a given $k$, there are $$\binom{n}{k}\cdot dd(I; n) \cdot b_{n-k}$$ ways to do this, so summing over all valid $k$ gives the first term of \eqref{eq1}. Next, suppose $w^{-1}(n+1) = m-1$. Then we must have a permutation of length $m-2$ to the right of $n+1$ with no double descents, and a permutation of length $n - (m-2)$ to the right of $w$ with an initial descent (which contributes to the double descent at $w_m$) but no double descents. There are $$\binom{n}{m-2}\cdot dd(\emptyset, n-m+2) \cdot\big(dd(\emptyset;n-m+2) - b_{n-m+2}\big)$$ such permutations, where the last term counts the number of permutations with an initial descent but no double descents. This gives the second term of \eqref{eq1}. Finally, suppose $w^{-1}(n+1) \in$ $[m-3]$. Then we can choose $0 \leq k \leq m-4$ elements to the right of $n+1$ to form a permutation with no double descents, and the remaining $n-k$ elements form a permutation with a double descent at $m-1-k$ (which is the $m$th spot in the entire permutation $w \in \mathfrak{S}_{n+1}$) and no initial descent. For a given $k$, there are $$\binom{n}{k}\cdot dd(\emptyset; k) \cdot c(\{m-1-k\};n-k)$$ ways to do this. Summing over all valid $k$ gives the third and final term of \eqref{eq1}.
	\end{proof}
	
	We do not have too much information on $c(I;n)$ so far, but because it consists of elements of $DD(I;n)$, it seems to follow a nice pattern, which is summed up in the following conjecture.
	
	\begin{conjecture}\label{c(n)estimate}
		The limit $\lim_{n \rightarrow \infty}\frac{c(\{m\}; n)}{dd(\{m\}; n)}$ exists for a fixed $m \in \N$. That is, we can estimate $c(\{m\}; n)$ as $dd(\{m\}; n)\cdot C(m)$, where $C(m)$ is some constant depending on $m$. Estimates for the first few values of $C(m)$ are: 
		\emph{\begin{center}
				\begin{tabular}{|c|c|c|c|c|c|c|c|c|c|} 
					\hline
					$m$ & 3 & 4 & 5 & 6 & 7 & 8 & 9 \\ 
					\hline
					$C(m)$ & 1 & 0.3941 & 0.6362 & 0.5056 & 0.5676 & 0.5359 & 0.5515 \\ 
					\hline
				\end{tabular}
		\end{center}}
	\end{conjecture}

	For a fixed $m$, the values of $\frac{c(\{m\}; n)}{dd(\{m\}; n)}$ oscillate as $n$ increases, and they appear to converge to some limit. Thus, we can accurately estimate $C(m)$ by averaging the values of $\frac{c(\{m\}; n)}{dd(\{m\}; n)}$ across $n$. We estimated the values in the table above by averaging $\frac{c(\{m\}; n)}{dd(\{m\}; n)}$ (for a fixed $m$) up to $n = 12$, as the values of $\frac{c(\{m\}; n)}{dd(\{m\}; n)}$ were already within $0.001$ of each other by $n=7$.
	
	We just computed some values of  for small $n$ and averaged them to produce the estimates of $C(m)$ in the table above.
	
	\begin{example}
		Using Theorem \ref{recursion1} and assuming Conjecture \ref{c(n)estimate}, we can estimate the value $dd(\{m\}; n)$. Suppose $n=9$ and $m=6$. Then the theorem gives 
	 	\begin{align*}
		dd(\{6\}, 9) = & \s \sum_{k=7}^8\dbinom{8}{k}\cdot dd(\{6\}; k)\cdot b_{8-k} + \dbinom{8}{4}\cdot dd(\emptyset;4)\cdot\big(dd(\emptyset;4)-b_4\big)\\
		+ & \s\sum_{k=0}^2\dbinom{8}{k}\cdot dd(\emptyset;k)\cdot c(\{5-k\};8-k) \\ = & \s \dbinom{8}{7}\cdot 426 \cdot 1 + \dbinom{8}{8}\cdot 2491\cdot 1 + \dbinom{8}{4}\cdot17\cdot(17-9) + \sum_{k=0}^2\dbinom{8}{k}\cdot dd(\emptyset;k)\cdot c(\{5-k\};8-k).
		\end{align*}
	\end{example}
	
	Using the estimation given by the conjecture, we can simplify this to
	\begin{align*}
	= & \s 15419 + \sum_{k=0}^2\dbinom{8}{k}\cdot dd(\emptyset;k)\cdot c(\{5-k\};8-k) \\ \approx &\s 15419 + \sum_{k=0}^2\dbinom{8}{k}\cdot dd(\emptyset;k)\cdot dd(\{5-k\}; 8-k) \cdot C(5-k) \\ = & \s 15419 + \dbinom{8}{0} \cdot 1 \cdot 2904 \cdot 0.6362 + \dbinom{8}{1} \cdot 1 \cdot 462 \cdot 0.3941 + \dbinom{8}{2} \cdot 2 \cdot 66 \cdot 1 \\ = & \s 22419.118.
	\end{align*}
	
	The actual value of $dd(\{6\}; 9)$ is 22419, so the estimate is off by $0.00053\%$. 
	
	\section{Rim Hooks}\label{rimhooks}
	 One important object associated with permutations, the rim hook, is brought up by considering permutations as \textit{rim hook tableaux}. Rim hooks are skew shapes that do not contain $2 \times 2$ squares. The following are examples of rim hooks:
	\vspace{0.3cm}
	\begin{center}
	    \ydiagram{1} \hspace{1cm} \ydiagram{1+2,2,1,1} \hspace{1cm}
	\ydiagram{1+1,1+1,2,1} \hspace{1cm} \ydiagram{1,1,1}
	\end{center}
	\vspace{0.3cm}
	
	We use the standard skew shape notation to represent these rim hooks. For example, the second rim hook from the above left is written as $(3,2,1,1)/(1)$, and the third rim hook from the above left is written as $(2,2,2,1)/(1,1)$. This notation is explained as follows: the first tuple of numbers represents the Young diagram which contains the rim hook. In the example of $(2,2,2,1)/(1,1)$, the numbers $(2,2,2,1)$ correspond to a Young diagram with 2 squares in the first row, 2 in the second, 2 in the third, and 1 in the fourth. The tuple of numbers after the slash represents the number of squares to remove from the rows of the specified Young diagram, in order to create the desired rim hook. So, the tuple $(1,1)$ in $(2,2,2,1)/(1,1)$ means we remove 1 square from the first row (starting on the left side of the row) of the aforementioned Young diagram, as well as 1 square from the second row, thus creating a rim hook. Also, the notation $|\mathfrak{r}|$ for a rim hook $\mathfrak{r}$ (and more generally, a skew shape) will denote the total number of squares in $\mathfrak{r}$.
	
	 A \textit{rim hook tableau} is formed by filling in the squares of a rim hook with the numbers $1$ through $n$, where $n$ is the number of squares in the rim hook, or the \textit{length} of the rim hook. A rim hook tableau also must satisfy the two following rules: for every two vertically adjacent squares, the upper square must contain the smaller number, and for every two horizontally adjacent squares, the left square must contain the smaller number. For example:
	\vspace{0.3cm}
	\begin{center}
	        \ytableaushort{\none312,45} \text{ is not a valid rim hook tableau, but } \ytableaushort{\none124,35} \text{ is valid.}
	\end{center}
	\vspace{0.3cm}
	
	Rim hooks can be used to encode the descent information of a permutation. This idea can be explained as follows: by reading a rim hook tableau from the bottom left to top right, following adjacent squares, we can reconstruct a permutation. For example, the above tableau on the right corresponds to the permutation [3,5,1,2,4] $\in \mathfrak{S}_5$. The rim hook of [3,5,1,2,4] precisely encodes a permutation in $\mathfrak{S}_5$ with a single descent at index 2. Any other permutation whose rim hook tableau has the same shape, such as [2,5,1,3,4], will have the same descents. Therefore, a rim hook of a certain shape will generate rim hook tableaux that correspond to permutations which all have the same descent set. This is how rim hooks can ``encode" descent sets (and analogously, double descent sets).
	
	 For example, the following are the rim hooks which generate permutations in $\mathfrak{S}_6$ with double descent set $\{2\}$: 
	\vspace{0.3cm}
	\begin{center}
	        \ydiagram{1+2,2,1,1} \hspace{1cm} \ydiagram{2+1,3,1,1} \hspace{1cm} \ydiagram{4,1,1}
	\end{center}
	\vspace{0.3cm}
	Some permutations with corresponding rim hook tableaux (to the rim hooks above, in that order) are [6,3,2,4,1,5], [5,4,1,2,6,3], and [4,3,2,1,5,6], all of which have a double descent at index 2.
	
	 We will use the notation $\mathcal{R}_I(n)$ to denote the set of all rim hooks of length $n$ which correspond to permutations with double descent set $I$. For example, the 3 rim hooks above are the elements of $\mathcal{R}_{\{2\}}(6)$. We can count the number of such rim hooks for singleton sets $I$ with the following formula.
	\begin{theorem}\label{rimhookrec}
	    $\#\mathcal{R}_{\{m\}}(n) = F_{n-m}F_{m-1}$, where $F_n$ is the $n$th Fibonacci number.
	\end{theorem}
	To prove this theorem, we need the following 2 propositions which give recurrences for $\#\mathcal{R}_I(n)$.
	
	\begin{proposition}\label{rhrecprop}
	Let $m = \emph{\text{max}}(I \cup \{0\})$. For $n \geq m + 3$, we have $\#\mathcal{R}_I(n) = \#\mathcal{R}_I(n-1) + \#\mathcal{R}_I(n-2).$
	\end{proposition}
	\begin{proof}
	All rim hooks must end in one of the two following shapes (i.e. these are their top right squares): 
	\vspace{0.3cm}
	\begin{center}
	        \begin{ytableau} \none & & \\
	        \none[\Ddots]
	        \end{ytableau}
	        \text{ or }
	        \begin{ytableau} \none &\\\none &\\
	        \none[\Ddots]
	        \end{ytableau}
	\end{center}
	\vspace{0.3cm}
	We will call rim hooks that end in the horizontal squares $H$-rim hooks and ones that end in vertical squares $V$-rim hooks. Now, suppose that $\mathcal{R}_I(n)$ contains $a$ $H$-rim hooks and $b$ $V$-rim hooks. To create a valid rim hook of $\mathcal{R}_I(n+1)$, we take rim hooks from $\mathcal{R}_I(n)$ and add an extra square, making sure not to create any additional double descents in the rim hooks. For example, the following shows valid and invalid extensions of a rim hook of $\mathcal{R}_{\{3\}}(7)$, where the shaded squares are the additional squares extending the rim hook (unshaded):
	\vspace{0.3cm}
	\begin{center}
	\begin{ytableau}
	\none & \none & \none & & *(green)\\
	\none & & & \\
	\none &\\
	&\\
	\end{ytableau}
	\begin{tabular}{c} 
    	\\\\\text{ is valid, but}
    \end{tabular}
	\begin{ytableau}
	\none & \none & \none & *(red)\\
	\none & \none & \none & \\
	\none & & & \\
	\none &\\
	&\\
	\end{ytableau}
	\begin{tabular}{c}
    	\\\\\text{is not.}
    \end{tabular}
	\end{center}
	\vspace{0.3cm}
	A valid extension of an $H$-rim hook can either be an extra square to the right or to the top of the top right end of the rim hook, so an $H$-rim hook can be respectively extended to a new $H$-rim hook and a new $V$-rim hook. For a $V$-rim hook, however, the only valid extension is the addition of one square to the right of the top right end of the rim hook, creating a new $H$-rim hook. Thus, if $\mathcal{R}_I(n)$ has $a$ $H$-rim hooks and $b$ $V$-rim hooks, then $\mathcal{R}_I(n+1)$ will have $a+b$ $H$-rim hooks and $a$ $V$-rim hooks, for a total of $\#\mathcal{R}_I(n+1) = 2a + b$ rim hooks. Applying this pattern again, we get $\#\mathcal{R}_I(n+2) = 3a+2b$, thus showing that the recursion $\#\mathcal{R}_I(n) = \#\mathcal{R}_I(n-1) + \#\mathcal{R}_I(n-2)$ holds.
	\end{proof}
	
	 This proposition shows that we can calculate any $\#\mathcal{R}_I(n)$ recursively, given the 2 initial values $\#\mathcal{R}_I(m+1)$ and $\#\mathcal{R}_I(m+2)$, where $m = \text{max}(I \cup \{0\})$. In order to prove Theorem \ref{rimhookrec} with the previous recursion, we need to determine initial values of $\mathcal{R}_{\{m\}}(n)$. The following proposition tells us what these initial values are.
	
	\begin{proposition}\label{initalRH}
	    For $m \geq 4$, we have $\#\mathcal{R}_{\{m\}}(m+1) = \#\mathcal{R}_{\{m-1\}}(m) + \#\mathcal{R}_{\{m-2\}}(m-1)$.
	\end{proposition}
	\begin{proof}
	    The argument in this proof is nearly the same as the one in Proposition \ref{rhrecprop}, except here we create extensions on the bottom left of a rim hook and not the top right. Also, in this scenario we will define $h$-rim hooks and $v$-rim hooks as rim hooks that \textit{start} with two horizontal or two vertical squares. Now, suppose $\mathcal{R}_{\{m\}}(m+1)$ consists of $a$ $h$-rim hooks and $b$ $v$-rim hooks. An extension of these rim hooks will increase the index of the descent by 1 and add 1 to the length of the rim hook, thereby creating an element of $\mathcal{R}_{\{m+1\}}(m+2)$. By the same argument as in Proposition \ref{rhrecprop}, $\mathcal{R}_{\{m+1\}}(m+2)$ will contain $a+b$ $h$-rim hooks and $a$ $v$-rim hooks, for a total of $2a + b$ elements. We also get $\#\mathcal{R}_{\{m+2\}}(m+3) = 3a + 2b$, thus showing the desired recursion is true.
	\end{proof}
	
	With Propositions \ref{rhrecprop} and \ref{initalRH}, we can now prove Theorem \ref{rimhookrec}.
	
	\begin{proof}[Proof of Theorem \ref{rimhookrec}] 
	    After brief computation we get that $\#\mathcal{R}_{\{2\}}(3) = 1$ and $\#\mathcal{R}_{\{3\}}(4) = 1$, so by Proposition \ref{initalRH}, we have $\#\mathcal{R}_{\{m\}}(m+1) = F_{m-1}$ for $m \geq 2$, where $F_n$ denotes the $n$th Fibonacci number. Now, for a fixed $m$, the smallest valid $n$ for which $\mathcal{R}_{\{m\}}(n)$ is defined is $m+1$, and the rim hooks in $\mathcal{R}_{\{m\}}(m+1)$ necessarily end in 3 vertical squares. Hence, there are no $H$-rim hooks (defined as in Proposition \ref{rhrecprop}) in $\mathcal{R}_{\{m\}}(m+1)$, so $\#\mathcal{R}_{\{m+1\}}(m+2)$ must equal $\#\mathcal{R}_{\{m\}}(m+1)$ because each $V$-rim hook in $\mathcal{R}_{\{m\}}(m+1)$ is extended to one new $H$-rim hook in $\#\mathcal{R}_{\{m+1\}}(m+2)$. Therefore, we have determined that $$\#\mathcal{R}_{\{m\}}(m+1) = \#\mathcal{R}_{\{m+1\}}(m+2) = F_{m-1}.$$
	    After applying the recursion from Proposition \ref{rhrecprop} to these initial values, we deduce Theorem \ref{rimhookrec}.
	\end{proof}
	
	 As we see, it is possible to calculate the size of any $\mathcal{R}_I(n)$ recursively, given two pre-computed initial values. However, there is a nicer non-recursive formula for the specific case $I = \emptyset$.
	
	\begin{theorem}\label{emptysetformula}
	Let $n \geq 2$, and let $F_n$ be the $n$th Fibonacci number. Then $\#\mathcal{R}_\emptyset(n) = \ds F_{n+1}$.
	\end{theorem}
	
	 Before we prove this theorem, we must first introduce the theory of \textit{minimal elements}. Define the \textit{height} of a rim hook (more generally, a Young diagram) to be the number of rows in the diagram. Then we define a \textit{minimal element of height h with double descent set I}, written as $\mu(I,h)$, as the rim hook of height $h$ that encodes double descent set $I$ and has the minimal number of squares possible. 
	
	For example, the following two rim hooks represent $\mu(\emptyset, 4)$ and $\mu(\{3\}, 5)$ respectively:
	
	\begin{center}
	        \ydiagram{2+1,1+2,2,1} \hspace{1cm} \ydiagram{3+1,2+2,1+2,1+1,2} \hspace{1cm}
	\end{center}
	
	 Minimal elements are useful because they allow us to quickly generate rim hooks by adding squares to the rows of a minimal element. The process of adding a square to a rim hook in general is as follows: to add a square to some row of a rim hook, just add a square to the right of the rightmost square in the specified row of the rim hook, and then shift all above rows to the right by 1.
	
	The following diagram demonstrates this process (added square is shaded):
	
	\vspace{0.3cm}
	\begin{center}
	\begin{ytableau}
    \none & \none & & &\\
	\none & &\\
	&\\
	\end{ytableau}
	\begin{tabular}{c} 
    	\\$\scalebox{1.5}{\ensuremath{\rightarrow}}$
    \end{tabular}
	\begin{ytableau}
    \none & \none & & &\\
	\none & & & *(green) \\
	&\\
	\end{ytableau}
    \begin{tabular}{c} 
    	\\$\scalebox{1.5}{\ensuremath{\rightarrow}}$
    \end{tabular}
	\begin{ytableau}
    \none & \none & \none & & &\\
	\none & & & \\
	&\\
	\end{ytableau}
	\end{center}
	\vspace{0.3cm}
	
	Now, notice that any rim hook can be decomposed into a minimal element, along with additional squares in some rows. For example, the above right rim hook is equivalent to $\mu(\emptyset, 3)$ with 2 added squares in the top row, 1 added square in the second row, and 1 added square in the bottom row. In the case that the double descent set of the rim hook is $\emptyset$, the double descent set of the minimal element will also be $\emptyset$. We formalize this argument as follows:
	
	\begin{proposition}\label{minimalbijection}
	Let $|\mu(I,h)|$ denote the number of squares in $\mu(I,h)$. Then we can construct all elements of $\mathcal{R}_\emptyset(n)$ of height $h$ by adding $n - |\mu(\emptyset,h)|$ squares to the rows of $\mu(\emptyset,h)$. Specifically, there is a bijection between the set of elements of $\mathcal{R}_\emptyset(n)$ of height $h$ and the set of all possible additions of $n - |\mu(\emptyset,h)|$ squares to $\mu(\emptyset,h)$.
	\end{proposition}
	
	\begin{proof}
	Suppose we have an arbitrary element  $\mathfrak{r}$ of $\mathcal{R}_\emptyset(n)$ of height $h$ for some $n$. Then, by the definition of minimal element, $\mu(\emptyset,h)$ must be contained within $\mathfrak{r}$. In particular, $\mathfrak{r}$ can be uniquely obtained from $\mu(\emptyset, h)$ by adding $|\mathfrak{r}| - |\mu(\emptyset,h)| = n - |\mu(\emptyset, h)|$ squares to $\mu(\emptyset, h)$ in the correct rows. 
	\end{proof}
	
	 For example, suppose we want to construct an element of $\mathcal{R}_\emptyset(8)$ with height 4. Then we take $\mu(\emptyset,4)$, and because this already has 6 squares in it, we just add the 2 remaining squares to any 2 not necessarily distinct rows. The following diagram shows how this process works (added squares are shaded):

	\vspace{0.3cm}
	\begin{center}
	\begin{ytableau}
    \none & \none & \\
	\none & & \\
	& \\
	\\
	\end{ytableau}
	\begin{tabular}{c} 
    	\\\\$\scalebox{1.5}{\ensuremath{\rightarrow}}$
    \end{tabular}
	\begin{ytableau}
    \none & \none & \\
	\none & & & *(green) \\
	& \\
	& *(green)\\
	\end{ytableau}
	\begin{tabular}{c} 
    	\\\\$\scalebox{1.5}{\ensuremath{\rightarrow}}$
    \end{tabular}
    \begin{ytableau}
    \none & \none & \none & \none &\\
	\none & \none & & & \\
	\none & & \\
	& \\
	\end{ytableau}
	\end{center}
	\vspace{0.3cm}
	
	To simplify notation for later, we will use the notation $\ext_n(\mathfrak{m})$ to denote the set of rim hooks of length $n$ generated by a minimal element $\mathfrak{m}$, i.e. \textit{extensions} of $\mathfrak{m}$. That is, elements of $\ext_n(\mathfrak{m})$ are created by adding $n - |\mathfrak{m}|$ extra squares to $\mathfrak{m}$ through the process of square-addition as shown above.

	Now that we have built up an understanding of minimal elements, we can proceed with the proof of Theorem \ref{emptysetformula}.
    
    \begin{proof}[Proof of Theorem \ref{emptysetformula}]
    By Proposition \ref{minimalbijection}, if $M$ represents the set of all possible minimal elements of length at most $n$, then $\#\mathcal{R}_\emptyset(n) = \sum_{\mathfrak{m} \in M}\#\ext_n(\mathfrak{m})$, because any element of $\mathcal{R}_\emptyset(n)$ is generated by the minimal element of the same height.
    
     Thus, we begin by determining all the minimal elements of $\mathcal{R}_\emptyset(n)$. We start with the simple cases: $\mu(\emptyset,1)$ is just a single square; $\mu(\emptyset, 2)$ is the Young diagram given by $(1,1)$, and $\mu(\emptyset, 3)$ is the skew shape given by $(2,2,1)/(1)$. More generally, all minimal elements of height greater than 2 (and for double descent set $\emptyset$) have a staircase shape, where the top and bottom rows have 1 square, and the middle rows all have 2 squares.
    
     Next, we determine the largest minimal element that can generate an element of $\mathcal{R}_\emptyset(n)$. Let $\mathfrak{m} = \mu(\emptyset, h)$ be the desired minimal element. Then $|\mathfrak{m}| = 2h - 2$, so the maximal $h$ such that $|\mathfrak{m}|\leq n$ is $H = \left \lfloor \frac{n+2}{2} \right \rfloor$.
    
     Now that we know all the minimal elements that generate elements of $\mathcal{R}_\emptyset(n)$, we are almost done. We can simplify the summation at the beginning of this proof as follows: $$\#\mathcal{R}_\emptyset(n) = \sum_{\mathfrak{m} \in M}\#\ext_n(\mathfrak{m}) = \sum_{k = 1}^H\#\ext_n(\mu(\emptyset,k))$$ because all the possible minimal elements are the ones of heights ranging from 1 to $H = \left \lfloor \frac{n+2}{2} \right \rfloor$.
    
     For a given height $h$, the value of $\#\ext_n(\mu(\emptyset,h))$ is the number of ways to distribute $n - |\mu(\emptyset,h)|$ additional squares among the $h$ rows of $\mu(\emptyset,h)$. This is commonly known as the number of weak $h$-compositions of $n - |\mu(\emptyset,h)|$, and this is given by the formula $$\dbinom{(n - |\mu(\emptyset,h)|) + h - 1}{h - 1} = \dbinom{n - (2h - 2) + h - 1}{h - 1} = \dbinom{n - h + 1}{h - 1}.$$
    
    Combining this with the previous summation, we get the following formula: $$\#\mathcal{R}_\emptyset(n) = \ds \sum_{k = 1}^H\#\ext_n(\mu(\emptyset,k)) = \sum_{k=1}^H\dbinom{n - k + 1}{k - 1}.$$
    
    It is well-known that this sum is equivalent to the $(n+1)$st Fibonacci number (see OEIS \cite{fibNum}), giving the desired result.
    \end{proof}
    
    \begin{example}
    Let us compute $\#\mathcal{R}_\emptyset(6)$ by using Theorem \ref{emptysetformula} and also by listing out the rim hooks individually. Theorem \ref{emptysetformula} gives $\#\mathcal{R}_\emptyset(6) = F_{7} = 13.$ Next, we list the elements of $\mathcal{R}_\emptyset(6)$:
    
    \vspace{0.3cm}
	\begin{center}
	        \scalebox{0.7}{\ensuremath{\ydiagram{6} \hspace{1cm} \ydiagram{4+1,5} \hspace{1cm} \ydiagram{3+2,4} \hspace{1cm} \ydiagram{2+3,3} \hspace{1cm} \ydiagram{1+4,2} \hspace{1cm} \ydiagram{5,1} \hspace{1cm} }}
	\end{center}
	\vspace{0.3cm}
	\begin{center}
	        \scalebox{0.7}{\ensuremath{\ydiagram{3+1,2+2,3} \hspace{1cm} \ydiagram{2+2,1+2,2} \hspace{1cm} \ydiagram{3+1,1+3,2} \hspace{1cm} \ydiagram{1+3,2,1} \hspace{1cm} \ydiagram{2+2,3,1} \hspace{1cm} \ydiagram{3+1,4,1} \hspace{1cm} \ydiagram{2+1,1+2,2,1}}}
	\end{center}
	\vspace{0.3cm}
    Indeed, there are 13 rim hooks in $\mathcal{R}_\emptyset(6)$, matching up with the value given by Theorem \ref{emptysetformula} as expected.
    \end{example}
	
	\section{Circular Permutations}\label{other}
	
	 Here we briefly mention the topic of \textit{circular permutations}. Intuitively, a circular permutation $w$ of length $n$ is just a permutation in $\mathfrak{S}_n$ ``wrapped-around''; that is, we read $w = [w_1, w_2, ..., w_n]$ from left to right, but when $w_n$ is reached, we just return back to $w_1$. This allows us to define double descents at all indices $1, 2,...,n$ and not just $2,3,...,n$. For example, a double descent at $n$ would mean $w_{n-1} > w_n > w_1$. Now, we formally define the set of circular permutations $\mathfrak{C}_n$ as follows. Define the \textit{rotation} map to be $\rho : \mathfrak{S}_n \xrightarrow{\sim} \mathfrak{S}_n$ which maps a permutation $w = [w_1, w_2, ..., w_n]$ to $[w_n, w_1, ..., w_{n-1}]$. Then, the set of equivalence classes of $\mathfrak{S}_n$ under the equivalence relation $w \sim \rho(w)$ is $\mathfrak{C}_n$.
	
	 When we discuss the double descents of a permutation $w \in \mathfrak{S}_n$, we mean double descents at the usual indices, $2, 3, ..., n-1$. However, if $w$ is an element of $\mathfrak{C}_n$, then double descents may also include indices $1$ and $n$.
	
	\begin{theorem}
		The number of permutations in $\mathfrak{C}_n$ with no double descents is equal to $b_{n-1}$.
	\end{theorem}
	\begin{proof}
		Each equivalence class defining $\mathfrak{C}_n$ has exactly one representative $w \in \mathfrak{S}_n$ satisfying $w_1 = n$. Therefore, we can count permutations in $\mathfrak{C}_n$ with no double descents by counting permutations in $\mathfrak{S}_n$ with first element $n$ that have no double descents (defined as usual, so at indices in $[2, n-1]$) and do not satisfy $w_{n-1} > w_n > w_1$ or $w_n > w_1 > w_2$. To construct such an element of $\mathfrak{S}_n$, we just take an element of $\mathfrak{S}_{n-1}$ with no double descents and no initial descent and put $n$ to the left of it. That is, if $u = [u_1, u_2, ..., u_{n-1}] \in$ $ \mathfrak{S}_{n-1}$ has no double descents and no initial descent, then $[n, u_1, u_2, ..., u_{n-1}]$ is the desired element of $\mathfrak{S}_n$. The no initial descent condition is required since $n > u_1$, as $u_1 \in [n-1]$, so this avoids a double descent at index $2$. Now we check that $[n, u_1, ..., u_{n-1}]$ has no double descents at all indices. A permutation of the form $[n, u_1, ..., u_{n-1}]$ has no double descents at indices $2,3,...,n-1$ by construction, and it also does not satisfy $w_{n-1} > w_n > w_1$ or $w_n > w_1 > w_2$ (i.e. has no double descents at indices $n$ and $1$) because $w_n < w_1$; $w_n \in [n-1]$ and $w_1 = n$, so $w_n$ must be less than $w_1$. Clearly, the number of such permutations is just the number of permutations in $\mathfrak{S}_{n-1}$ with no double descents and no initial descent, $b_{n-1}$.
	\end{proof}
	
	\section{Conjectures}\label{conjs}
	
	 All of the following conjectures come from observing patterns in computer-produced data tables of values of $dd(I;n)$ for various $I$ and $n$. 
	 
	 \begin{conjecture}
		The values $\{dd(\{i\}; n)\}_{i=1}^n$ are asymptotically equidistributed. Namely, for fixed $0 < \alpha < \beta < 1$, we have $\ds \sum_{\alpha n < i < \beta n}dd(\{i\}; n) \sim (\beta - \alpha)\sum_{i=2}^{n-1}dd(\{i\};n)$.
	\end{conjecture}
	
	\begin{remark*}
	    This conjecture can be intuitively understood, as when a permutation becomes extremely long (i.e. for large $n$), the probability there is a double descent at index $k$ should be nearly the same as the probability of a double descent at index $k+1$.
	\end{remark*}

	\begin{conjecture}
		Given a fixed $n \in \N$, the numbers $dd(\{i\};n)$ for $2 \leq i < \left \lceil \frac{n}{2} \right \rceil$ follow a ``down up down up" pattern. Namely, $dd(\{i\}; n) > dd(\{i + 1\}; n)$ if $i$ is even, and $dd(\{i\}; n) < $ $dd(\{i + 1\}; n)$ if $i$ is odd.
	\end{conjecture}
	
	\begin{remark*}
	    This conjecture is very unexpected, as it seems to hold for all values of $n$ (numerically verified for some $n$). In particular, the ``down up down up" pattern persists even as the values of $dd(\{i\};n)$ approach uniform distribution.
	\end{remark*}
	
	\begin{conjecture}
		Let $n,i \in \N$ such that $i < \left \lceil \frac{n}{2} \right \rceil - 1$. If $i$ is even, we have $\frac{dd(\{i\}; n)}{dd(\{i+1\}; n)}>$ $\frac{dd(\{i+2\}; n)}{dd(\{i+3\}; n)}$, and if $i$ is odd, we have $\frac{dd(\{i\}; n)}{dd(\{i+1\}; n)}<$ $\frac{dd(\{i+2\}; n)}{dd(\{i+3\}; n)}$.
	\end{conjecture}
	
	\begin{remark*}
	    This conjecture illustrates how the values of $dd(\{i\};n)$ approach uniform distribution as $n$ becomes large. The values of the successive ratios $\frac{dd(\{i\}; n)}{dd(\{i+1\}; n)}$, which are approaching 1 as $dd(\{i\};n)$ reaches uniform distribution, are strictly decreasing toward 1 for successive even $i$, while these ratios are strictly increasing toward 1 for successive odd $i$.
	\end{remark*}
	
	\begin{conjecture}
		For fixed $i,j \in \Z_{\geq 2}$, the limit $\lim_{n \rightarrow \infty}\frac{dd(\{i\}; n)}{dd(\{j\}; n)}$ exists and is a positive number.
	\end{conjecture}
	
	\begin{remark*}
	This conjecture highlights a major difference between descents and double descents. According to Diaz-Lopez et al. \cite{DLHIOS}, $d(\{i\};n)$ is a polynomial of degree $i$, so $\lim_{n \rightarrow \infty}\frac{d(\{i\}; n)}{d(\{j\}; n)}$ is either $0$ or $\infty$ when $i \neq j$, whereas the corresponding limit for double descents is always a positive number.
	\end{remark*}
	
	In fact, we can generalize this conjecture:
	
	\begin{conjecture}
	    Let $I, J \subset \Z_{\geq 2}$ be two finite sets such that $dd(I;n) > 0$ and $dd(J;n) > 0$ for all but finitely many $n$. Then the limit $\lim_{n \to \infty}\frac{dd(I;n)}{dd(J;n)}$ exists and is a positive number.
	\end{conjecture}
	
	The following graphs show values of $\frac{dd(I;n)}{dd(J;n)}$ plotted with respect to $n$ for various $I$ and $J$:
	
	\begin{figure}[h]
    \centering
    \subfloat[$I = \{5\}$ and $J = \{2,3,4\}$]{{\includegraphics[width=7cm, keepaspectratio]{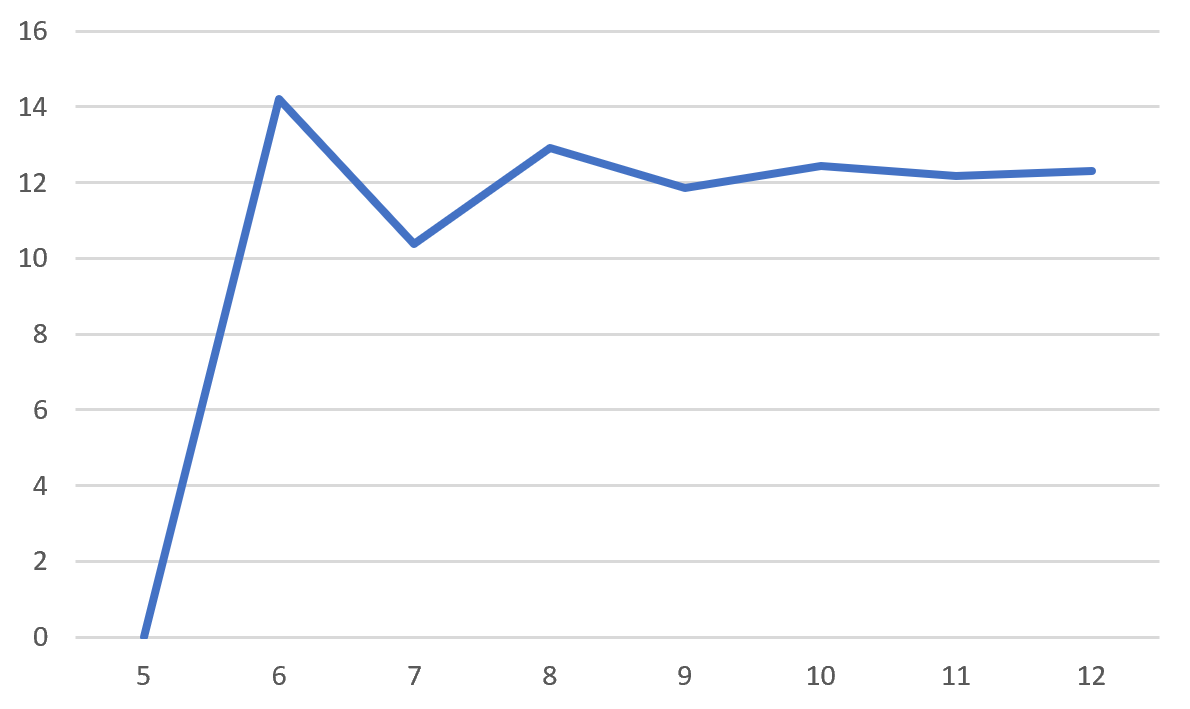}}}
    \qquad
    \subfloat[$I = \{3\}$ and $J = \{2,3\}$]{{\includegraphics[width=7cm, keepaspectratio]{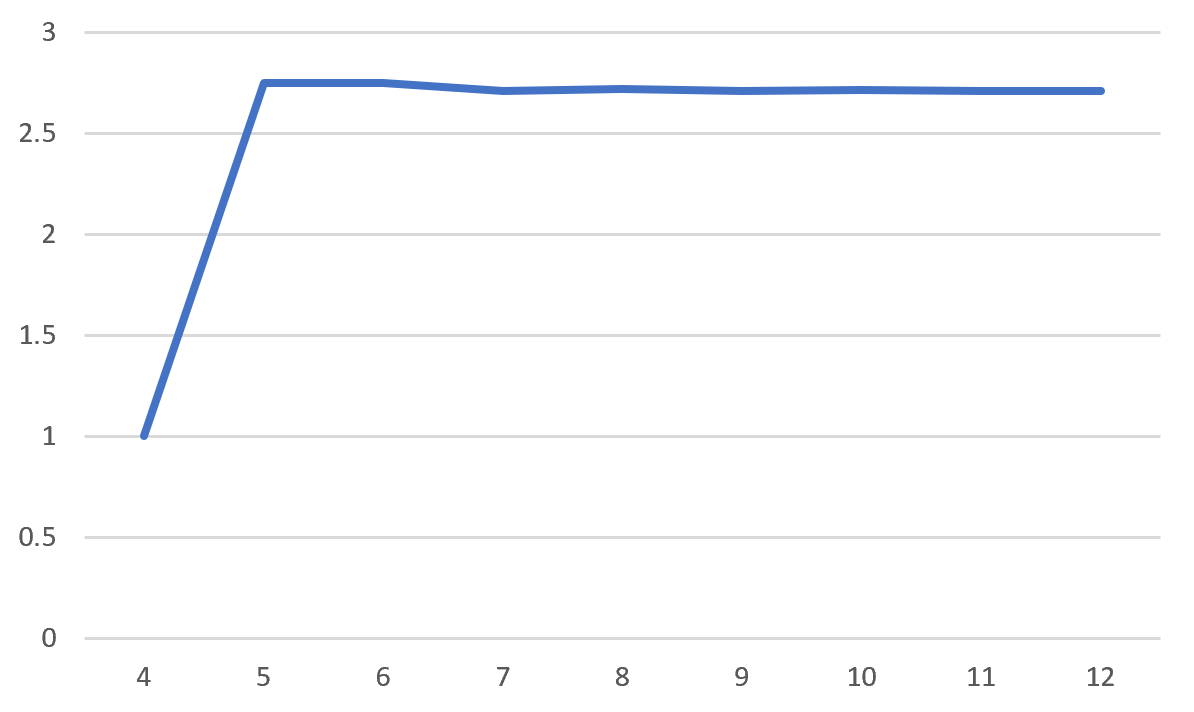}}}
    \newline
    \centering
    \subfloat[$I = \emptyset$ and $J = \{2,5\}$]{{\includegraphics[width=7cm, keepaspectratio]{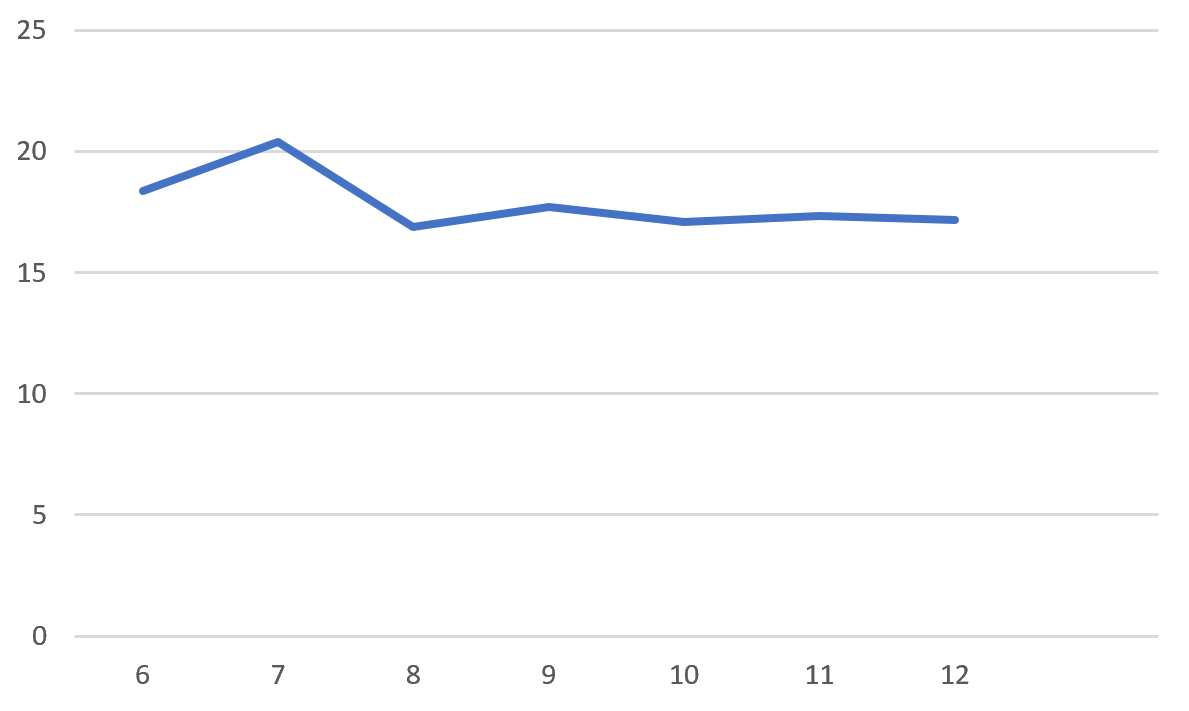}}}
    \qquad
    \subfloat[$I = \{2\}$ and $J = \{4\}$]{{\includegraphics[width=7cm, keepaspectratio]{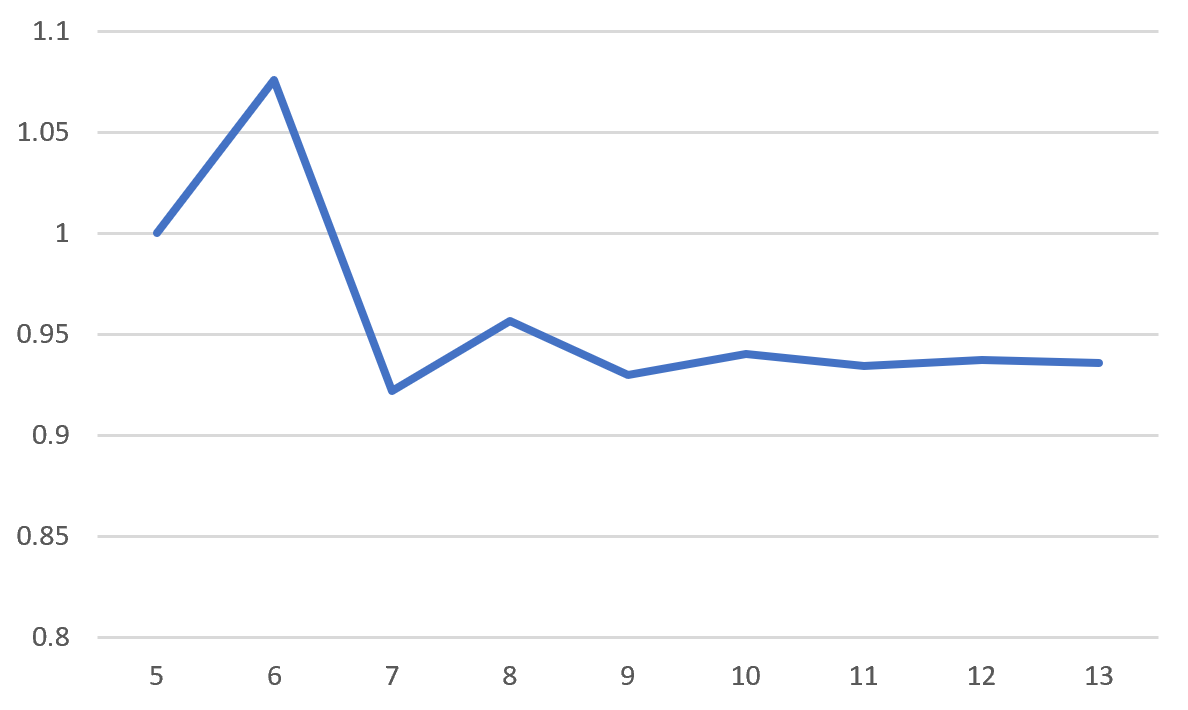}}}
    \label{fig:example}
    \end{figure}
	
	Each graph demonstrates that $\frac{dd(I;n)}{dd(J;n)}$ converges; in particular, each ratio converges alternately.

    \section{Future Work}
	 It might be possible to establish lower and upper bounds on $dd(I;n)$ by using Naruse's hook-length formula \cite{naruse} for skew shapes as well as Proposition \ref{rhrecprop}. Let $I$ be a double descent set. By definition of $\mathcal{R}_I(n)$, we have $$dd(I;n) = \sum_{\mathfrak{r}\in\mathcal{R}_I(n)}f^{\mathfrak{r}},$$ where $f^{\mathfrak{r}}$ is the number of rim hook tableaux of $\mathfrak{r}$. Then, we have the following bounds:$$ \inf_{\mathfrak{r}\in\mathcal{R}_I(n)}f^{\mathfrak{r}}\cdot\#\mathcal{R}_I(n) \leq dd(I;n) \leq \sup_{\mathfrak{r}\in\mathcal{R}_I(n)}f^{\mathfrak{r}}\cdot\#\mathcal{R}_I(n).$$
	
	 With the recursion given in Proposition \ref{rhrecprop}, one can determine $\#\mathcal{R}_I(n)$ as long as the initial conditions for the recursion are computed. For example, we have already determined the initial conditions for singleton double descent sets, allowing us to formulate Theorem \ref{rimhookrec}.
	
	 To evaluate $ \inf_{\mathfrak{r}\in\mathcal{R}_I(n)}f^{\mathfrak{r}}$ and $\sup_{\mathfrak{r}\in\mathcal{R}_I(n)}f^{\mathfrak{r}}$, one might be able to use Naruse's hook-length formula, which is as follows: $$f^{\lambda/\mu} = |\lambda/\mu|!\Bigg[\sum_{D \in \mathcal{E}(\lambda/\mu)}\Bigg(\prod_{c \in \lambda/D}\dfrac{1}{h(c)}\Bigg)\Bigg],$$ where $\lambda/\mu$ is a skew shape, and $\mathcal{E}(\lambda/\mu)$ is the set of \textit{excited diagrams} of $\lambda/\mu$, and $h(c)$ is the hook-length of a square $c$ as calculated in $\lambda$. More explanation on this formula can be found in the literature \cite{idk2, idk1}.
    
    \begin{flushleft}
	\section{Acknowledgments}
	 The author would first like to thank Pakawut Jiradilok (MIT) for his guidance and advice in this research. The author would also like to thank Yongyi Chen (MIT) and Dr. Tanya Khovanova (MIT) for their proofreading of this paper.	The author would finally like to thank Professor Pavel Etingof (MIT), Dr. Slava Gerovitch (MIT), and Dr. Khovanova for providing this research opportunity at the MIT PRIMES program. The author is especially grateful for Dr. Khovanova's suggestions regarding double descents.
	 
	 \textit{This material is based upon work supported by the National Science Foundation under Grant No. 1916120.}
	 \end{flushleft}
	\newpage
	\begin{flushleft}
	
	\end{flushleft}
\end{document}